\documentclass[12pt]{article}
\usepackage{graphics}
\usepackage{epsfig}
\usepackage{amsmath}
\usepackage{amsfonts}

\font\elevensf=cmss10 scaled\magstephalf

\newtheorem{theorem} {\bf THEOREM}[section]

\newtheorem{counter example} {\bf Counter Example}[section]

\newtheorem{definition} {\bf DEFINITION}[section]

\def\CC{{\rm \kern.24em \vrule width.02em height1.4ex depth-.05ex \kern-.26emC}}

\def\TagOnRight

\def\AA{{it I} \hskip-3pt{\tt A}}

\def\QQ{\rlap {\raise 0.4ex \hbox{$\scriptscriptstyle |$}} {\hskip -0.1em Q}}

\newcommand{\be} {\begin{equation}}
\newcommand{\ee} {\end{equation}}
\newcommand{\bea} {\begin{eqnarray}}
\newcommand{\eea} {\end{eqnarray}}
\newcommand{\Bea} {\begin{eqnarray*}}
\newcommand{\Eea} {\end{eqnarray*}}

\catcode`\@=11

\def\theequation{\@arabic{\c@section}.\@arabic{\c@equation}}

\catcode`\@=12
\catcode`\@=11
\title{BV regularity near the interface for nonuniform convex discontinuous flux}
\author{ Shyam Sundar
Ghoshal\footnote{ E-mail:shyam.ghoshal@gssi.infn.it.
 Acknowlegements: The author would like to 
thank the anonymous referee for providing fruitful suggestions in the present article.}\\\\
Gran Sasso Science Institute, \\Viale Francesco Crispi, 7,\\ 67100 L’Aquila, Italy.}

\date{}

\begin{document}

\maketitle

\begin{abstract} In this paper, we discuss the total variation bound for the solution of scalar conservation laws with discontinuous
flux. We prove the smoothing effect of the equation forcing the $BV_{loc}$ solution 
near the interface for $L^\infty$ initial data without the assumption on the uniform convexity of the fluxes made as in \cite{ARGS, ghoshaljde}.
The proof relies on the method of characteristics and the explicit formulas. 
\end{abstract}


\section{Introduction}

Let us consider the following conservation laws  with discontinuous flux,
\begin{equation}\label{bv11}
 \begin{array}{lllll}
  u_t+F(x,u)&=0\\
  u(x,0)&=u_0(x),
 \end{array}
\end{equation}
 the flux $F$ is given by, $F(x,u)=H(x)f(u)+(1-H(x))g(u)$, $H$ is the Heaviside function.
Throughout this paper we assume the fluxes $f,g$ to be $C^2(\mathbb{R}),$ strictly convex, superlinear growth (see definition \ref{lt}) 
and $u_0\in L^\infty$.
 We denote by $f_+^{-1},g_+^{-1}$ to be the inverses of the increasing part of $f$, $g$ respectively and similarly $f_-^{-1},g_-^{-1}$,
 the inverses of the decreasing part of  $f,g$ respectively. 
 Let $\theta_f, \theta_g$ be the unique minimums of the fluxes $f$ and $g$ respectively. By uniform convexity of the fluxes $f,g$, we 
 mean $f,g\in C^2$ and there exists a positive
 constant $\alpha$ such that $f^{\prime\prime}>\alpha,\ g^{\prime\prime}>\alpha.$

\par The first order partial differential equation of type (\ref{bv11}) has many applications namely, modeling gravity, continuous sedimentation in
a clarifier-thickener unit, petroleum industry, traffic flow on a highway, semiconductor industry. For more details, one can see
\cite{burKarlKlinRis, BurgKarRis, BurKarRisTow, BurKarTow, diehlSIAM, diehlengrg, geo, mochon} and the references therein.

\par It is well known that even if the flux and the initial data are sufficiently smooth
the global classical solution of scalar conservation laws, does not exist always, which allows to define the weak notion of the solution.
In general, there might be infinitely many weak solutions even for smooth flux. Moreover, 
the well-posedness theory for the Cauchy problem for scalar conservation laws with Lipschitz flux function 
was completely settled by Kruzkov \cite{kruzkov}. 

\par In past few decades, the Cauchy problem for  conservation laws with discontinuous flux of the type (\ref{bv11}) has been well studied, 
it has been tackled in several ways. For convergence analysis of several important  numerical schemes we refer to
\cite{siam, jhde, mathcomp, And2, karlsen, burger, BurgKarRis, BurKarRisTow, towers}. The solution of (\ref{bv11}) can be achieved by  
vanishing viscosity limit \cite{karlRisTowers,And1}, front tracking (\cite{burKarlKlinRis,gimseresebro, gimse, KlinRis}).
On the other hand, one can obtain the solution of (\ref{bv11}) also via Hamilton-Jacobi equation 
\cite{Kyoto, jde, ostrov}, which is the one we exploit in the present paper. 
Consider the following Hamilton-Jacobi equation  
\begin{equation}\label{bv12}
\begin{array}{lllll}
v_t+g(v_x)&=&0 \ \ \mbox{if}\  x<0, t>0,\\
v_t+f(v_x)&=&0 \ \ \mbox{if} \ x>0, t>0,\\
v(x,0)&=&v_0 \  \mbox{if}\ x\in\mathbb{R}.
\end{array}
\end{equation}
\par In \cite{jde}, they dealt with the Hamilton-Jacobi equation (\ref{bv12}) and proved the existence of the global Lipschitz solution
of (\ref{bv12}). Then $u=v_x$, solves the  discontinuous conservation laws (\ref{bv11}) 
with the initial data $v_{0_x}=u_0$. Also they proved the existence of infinitely many stable semigroups of entropy solutions based on
$(A,B)$ interface entropy condition (see definition \ref{definitionInter}). For any such choice of $A,B$, they have given an explicit representation of the
entropy solution for the conservation law (\ref{bv11}). Throughout our paper we use the $(A,B)$ entropy solution 
(see definition \ref{definitionSol}) obtained as in \cite{Kyoto, jde}.

\par Despite the fact that the subject is well studied, the total variation bound near the interface remained unsolved for quite a long time.
In \cite{karlsen}, they observed that the solution stays in BV, away from the  interface $x=0$.
Recently, in \cite{ARGS,ghoshaljde} the question regarding the total variation bound near
interface has been answered. 
In general, BV regularity of the solution is an extremely important phenomenon in the theory of conservation laws. 
For the case when $f=g$ (not necessarily convex), total variation diminishing (TVD) property holds, that is to say,
$TV(u(\cdot, t))\leq TV(u_0(\cdot))$, for all $t>0.$ Definitely, one cannot expect to have similar property in the case when $f\neq g,$ because 
of the fact that constant initial data may lead to a non constant solution. 

\par The first breakthrough results regarding the total variation bound have been obtained in \cite{ARGS}. 
First it was observed in \cite{ARGS} that the critical points plays a key role for the existence and nonexistence of 
total variation bound. It has been shown in \cite{ARGS} that if the connections $(A,B)$ avoid the critical points $(\theta_g,\theta_f)$
and $u_0\in BV$, then 
the BV regularity holds. 
They observed in \cite{ARGS} that $u_0\in BV$ is not enough, one needs to assume also
$f^{-1}_+g(u_0),g^{-1}_-f(u_0)\in BV$ in order to prove that the solution is  BV  near the interface. 
Also they have constructed a counter example 
by rarefaction waves and shock waves separated by constant states near the interface 
to prove that total variation of the solution at time $t=1$ may blow up even if $u_0\in BV.$
Hence the assumption $f^{-1}_+g(u_0),g^{-1}_-f(u_0) \in BV$
are important. 

\par Later on in \cite{ghoshaljde}, it has been proved that 
if the connections $(A,B)$ avoid the critical points $(\theta_g,\theta_f)$ and $f^{\prime\prime},g^{\prime\prime}\geq \alpha>0,$ then $u\in BV$ for $u_0\in L^\infty.$
Also a very strong and surprising   result has been obtained in \cite{ghoshaljde}
that if the lower heights of both the fluxes are same that is if $f(\theta_f)=g(\theta_g),$ then the solution stays
in BV near the interface,
even if the fluxes are not uniformly convex or even if $u_0\notin BV.$ On the other hand, 
if $f(\theta_f)\neq g(\theta_g),$ then the results does not hold in general (for instance the counter example in \cite{ARGS} at $t=1$).
If  $f(\theta_f)\neq g(\theta_g),$ then one needs to put an extra assumption that
Supp $u_0$ is compact, then the result holds true for uniformly convex flux, but for large time.
One can not avoid to 
assume that the initial data $u_0$ is compactly supported due to the counter example (in \cite{ghoshaljde}),
which shows that even for uniformly convex fluxes,  $u(x,T_n)\notin BV$ while $u_0\in BV$, for all $n$
and $\lim\limits_{n\rightarrow\infty}T_n=\infty$. 

\par When $f=g,$ we have the following Lax-Oleinik formula: 
\begin{theorem}\label{thm1.1} (Lax-Oleinik formula) Let the initial data $u_0\in L^\infty(\mathbb{R})$. If the flux  $f$  is $C^2$, uniformly convex
and of super linear growth. Then 
there exists a function  $y (x, t)$ such that   
\begin{enumerate}
 \item $x\mapsto y(x,t)$ is non decreasing.
 \item  For a.e. $(x,t)\in \mathbb{R}\times (0,\infty)$, the solution $u$ of (\ref{bv11})  is given by  
 \begin{align} u (x,t)  = (f^\prime)^{-1}\left(\frac{x-y (x,t)}{t}\right).  \label{eq1.2} \end{align}  
\end{enumerate}\end{theorem}

\par An immediate  observation from theorem \ref{thm1.1}  is the following:

\noindent If $f^{\prime\prime}\geq \alpha>0$ and   then for all $ t > 0, \  x 
\mapsto u (x, t)$ is  in $ B V_{loc} (\mathbb{R}),$ even if $u_0\notin BV.$ When the flux is not uniformly convex, in the case  $f=g$,
BV regularity does not hold for $u_0\in L^\infty$, even 
for large time (one can see a counter example in \cite{Ssbv}, though $f^\prime(u)\in BV_{loc}$). For finer properties of characteristics in the case
$f=g$,
we refer to \cite{Ssbv,Ssh,ghoshal,duke,lax,Serre}.

\par When $f\neq g$,  Lax-Oleinik type formula still holds: 

\begin{theorem} (Adimurthi, Gowda \cite{Kyoto}, \cite{jde})
 Let $u_0\in L^{\infty}(\mathbb{R})$, then there exists an entropy solution $u$ of (\ref{bv11})  with $u\in L^{\infty}(\mathbb{R}).$  Also there exist 
 Lipschitz continuous functions $R(t)\geq 0,L(t)\leq 0.$
 and  monotone functions $y_{\pm}(x,t), \ t_{\pm}(x,t)$ such that 

 For  a fixed $t>0,$ 
 \begin{itemize}
  \item  $t\geq t_+(x,t)\geq 0$ is a non-increasing function of $x$ in $ [0,R(t)), $
  \item $ y_+(x,t)\geq 0$ is a non-decreasing function of $x$ in $ [R(t),\infty), $
   \item $t\geq t_-(x,t)\geq 0$ is a non-decreasing function of $x$ in $ (L(t),0], $
    \item $ y_+(x,t)\leq 0$ is a non-decreasing function  of $x$ in $ (-\infty,L(t)). $
 \end{itemize}
\begin{eqnarray}
 u(x,t)=
\left\{\begin{array}{lll} (f^{\prime})^{-1}\left(\frac{x-y_+(x,t)}{t}\right) &\mbox{ if }& x\geq R(t),\\
(f^{\prime})^{-1}\left(\frac{x}{t-t_+(x,t)}\right) &\mbox{ if }& 0\leq x<R(t).
\end{array}\right.\label{r2.12}
\end{eqnarray}

\begin{eqnarray}
 u(x,t)=
\left\{\begin{array}{lll}
 (g^{\prime})^{-1}\left(\frac{x-y_-(x,t)}{t}\right) &\mbox{ if }& x\leq L(t), \\
(g^{\prime})^{-1}\left(\frac{x}{t-t_-(x,t)}\right) &\mbox{ if}& L(t)<x<0.
\end{array}\right.\label{r2.13}
\end{eqnarray}

\end{theorem}

\par One can ask the similar question as above that under which assumptions on 
$f,g,u_0$ can one expect to have $u \in BV_{loc}(-\infty,L(t))$, $u\in BV_{loc}(L(t),0)$, $ u\in BV_{loc}(0,R(t))$, $u\in BV_{loc}(R(t),\infty)?$ 
In this present article we prove that  $ u\in BV_{loc}(L(t),0)$, $u\in BV_{loc}(0,R(t))$ with nonuniform convex flux and $u_0\in L^\infty.$
that is to say  we relax  $f^{\prime\prime},g^{\prime\prime}\geq \alpha>0$ and avoid $u_0\in BV$
to prove the existence of the BV regularity near the 
interface.  For the case when the connections avoid
the critical points, we prove that the solution stays in BV near the interface, even if the flux is not uniformly convex
(e.g. $f(u)=u^4, g(u)=u^6+1$) and even 
if the initial data is highly oscillatory ($u_0\in L^\infty$). Also we prove that if one allows the connection to be the critical points, 
then similar results hold true without uniform convexity of fluxes and only with $u_0\in L^\infty$, but for large time. 
The result in this paper is very surprising because, even if $f=g,$ but not uniformly
convex then in general there does not exist any region where the solution stays BV for $L^\infty$ initial data, in other words, for $f=g$ and 
nonuniform convex fluxes
one can always choose some $u_0\in L^\infty$ such that $u(\cdot,t)\notin BV(K)$, where $K$ is an interval in $\mathbb{R}$ and $t>0$ (one can see the 
details in \cite{Ssbv}). This happens due to the lack of Lipschitz continuity of $(f^\prime)^{-1}$. Here in this article we prove that
near the interface
i.e, in the region $(L(t),0)\cup (0,R(t))$, either the solution does not oscillates too much or $(f^\prime)^{-1},(g^\prime)^{-1}$ are Lipschitz continuous.
Note that for  $x\in(L(t),0)\cup (0,R(t))$, the characteristic passing through the point $(x,t)$,
bended before time $t$ if and only if 
$f\neq g.$ 
The paper present a special case in which one can prove BV regularity for the solution near the interface 
(starting from $L^\infty$ data) without enforcing  the assumption on uniform convexity of the flux.
However, the price to pay in order to relax the hypothesis on uniform convexity is not negligible.
Our proof relies on the explicit Lax-Oleinik type formulas obtained in \cite{Kyoto, jde} and 
the finer analysis of the characteristics curves \cite{Kyoto, jde, Ssh}.


\par In order to make the present article self contained, we describe enough prerequisites in section 2 before presenting the 
main results in section 3.

\section{Preliminaries}  We recall some definitions and known results from \cite{ARGS, Ssh, Kyoto, jde, 
ghoshaljde}.
\begin{definition} \textbf{Weak solution of} (\ref{bv11}): $u$ is said to be weak solution of (\ref{bv11}) if 
$u\in L^\infty_{\mbox{loc}}(\mathbb{R}\times \mathbb{R}_+)$ and it satisfies the following integral
equality, for all $\phi\in C^\infty_{0}(\overline{{\mathbb{R}\times \mathbb{R}_+}})$
\begin{equation}\label{bv21}
 \int\limits_{0}^{\infty} \int\limits_{-\infty}^{\infty}\left(u\frac{\partial \phi}{\partial t}+(H(x)f(u)+(1-H(x)g(u))
 \frac{\partial \phi}{\partial x}\right)dxdt + \int\limits_{-\infty}^{\infty}u_0(x)\phi(x,0)dx=0.
\end{equation}
It is immediate to check that $u$ satisfies (\ref{bv21}) if and only if  $u$ satisfies the following in weak sense 
\begin{align}\label{bv4}
 \begin{array}{lllll}
   u_t+g(u)_x&=0 \ \ &\mbox{if}\  x<0, t>0,\\
  u_t+f(u)_x&=0 \ \ &\mbox{if} \ x>0, t>0,\\
  u(x,0)&=u_0 \  &\mbox{if}\ x\in\mathbb{R}.
 \end{array}
\end{align}

\end{definition}
\noindent\textbf{Rankine-Hugoniot condition at interface}: Let us denote $u(0+,t)=\lim\limits_{x\rightarrow 0+}u(x,t)$ and
$u(0-,t)=\lim\limits_{x\rightarrow 0-}u(x,t)$. Then at $x=0$, $u$ satisfies the following R-H condition
\begin{equation}\label{bv22}
 f(u(0+,t))=g(u(0-,t)), \ \mbox{a.e.}\ t>0.
\end{equation}

\begin{definition} \textbf{Interior entropy condition}:  A weak solution $u$ of (\ref{bv4})
is said to  satisfy the interior entropy condition if
\begin{equation}\label{bv23}
 \begin{array}{llll}
 \frac{\partial \phi_1(u)}{\partial t} + \frac{\partial \psi_1(u)}{\partial x}\leq 0 \ \mbox{for} \ x>0,t>0,\\
  \frac{\partial \phi_2(u)}{\partial t} + \frac{\partial \psi_2(u)}{\partial x}\leq 0 \ \mbox{for} \ x<0,t>0,
 \end{array}
\end{equation}
in the sense of distributions, where $(\phi_i,\psi_i)$ are the convex entropy pairs such that 
$(\psi_1^\prime(u),\psi_2^\prime(u))=(\phi_1^\prime(u)f^\prime(u), \phi_2^\prime(u)g^\prime(u))$.
\end{definition}

\begin{definition}
 ($\mathbf{(A,B)}$ \textbf{Connection}): Let $(A,B)\in\mathbb{R}^2$,  is called a connection if it satisfies the following \\
 (i) $g(A)=f(B)$.\\
(ii) $ g'(A)\leq0, f'(B)\geq0.$
\end{definition}

\begin{definition}\label{definitionInter}
 \noindent \textbf{$\mathbf{(A,B)}$ Interface entropy condition}: Let  $u\in L^{\infty}_{\textrm{loc}}{(\mathbb{R}\times\mathbb{R}_+)}$ such that $u(0\pm,t)$ 
 exist a.e. $t>0$.
 Define $I_{AB}(t)$ by 
\begin{equation}
 (g(u(0-,t))-g(A))\textrm{sign}(u(0-,t)-A)-(f(u(0+,t))-f(B))\textrm{sign}(u(0+,t)-B).
\end{equation}
Then $u$ is said to satisfy $(A,B)$ Interface entropy condition if for a.e. $t>0$,
\begin{equation}
 I_{AB}(t)\geq 0.\label{interface}
 \end{equation}
 When $A=\theta_g$ or $B=\theta_f$, then (\ref{interface}) reduces to 
  \begin{equation}
 \textrm{meas}\big\{t:f'(u(0+,t))>0,g'(u(0-,t))<0\big\}=0\label{interface2}.
\end{equation}

\end{definition}

\begin{definition}\label{definitionSol}
 \noindent \textbf{$\mathbf{(A,B)}$ Interface entropy solution}: $u\in L^{\infty}_{\textrm{loc}}{(\mathbb{R}\times\mathbb{R}_+)}$ is said to be 
 the  $(A,B)$ Interface entropy solution of  (\ref{bv11}), if it satisfy (\ref{bv21}), (\ref{bv23}) and (\ref{interface}).
\end{definition}

\begin{definition} \textbf{Control curve}: Let $0\leq t$ and $\gamma\in C([0,t],\mathbb{R})$. Then   $\gamma $ is called a control curve if the 
following holds.
\begin{enumerate}
\item [(i).] $\gamma(t)=x$.  
\item[(ii).]  $\gamma$  consists of at most three linear curves and each segment lies completely in either $x
\geq 0$ or $x \leq 0 \,.$ 

\item[(iii).] Let $0 = t_3 \leq t_2 \leq t_1 \leq t_0 = t$ be such that for $i = 1, 2, 3, \quad \gamma_i =
 \gamma |_{[t_i, t_{i - 1}]}$ be the linear parts of $\gamma$.  If $\gamma$ consists of three linear curves then $\gamma_2 = 0$ 
and $\gamma_1,\gamma_3>0$ or $\gamma_1,\gamma_3<0.$
\end{enumerate}

\begin{figure}[http]
\label{fig1}
 \begin{center}
\includegraphics[width=4in, height=3.2in]{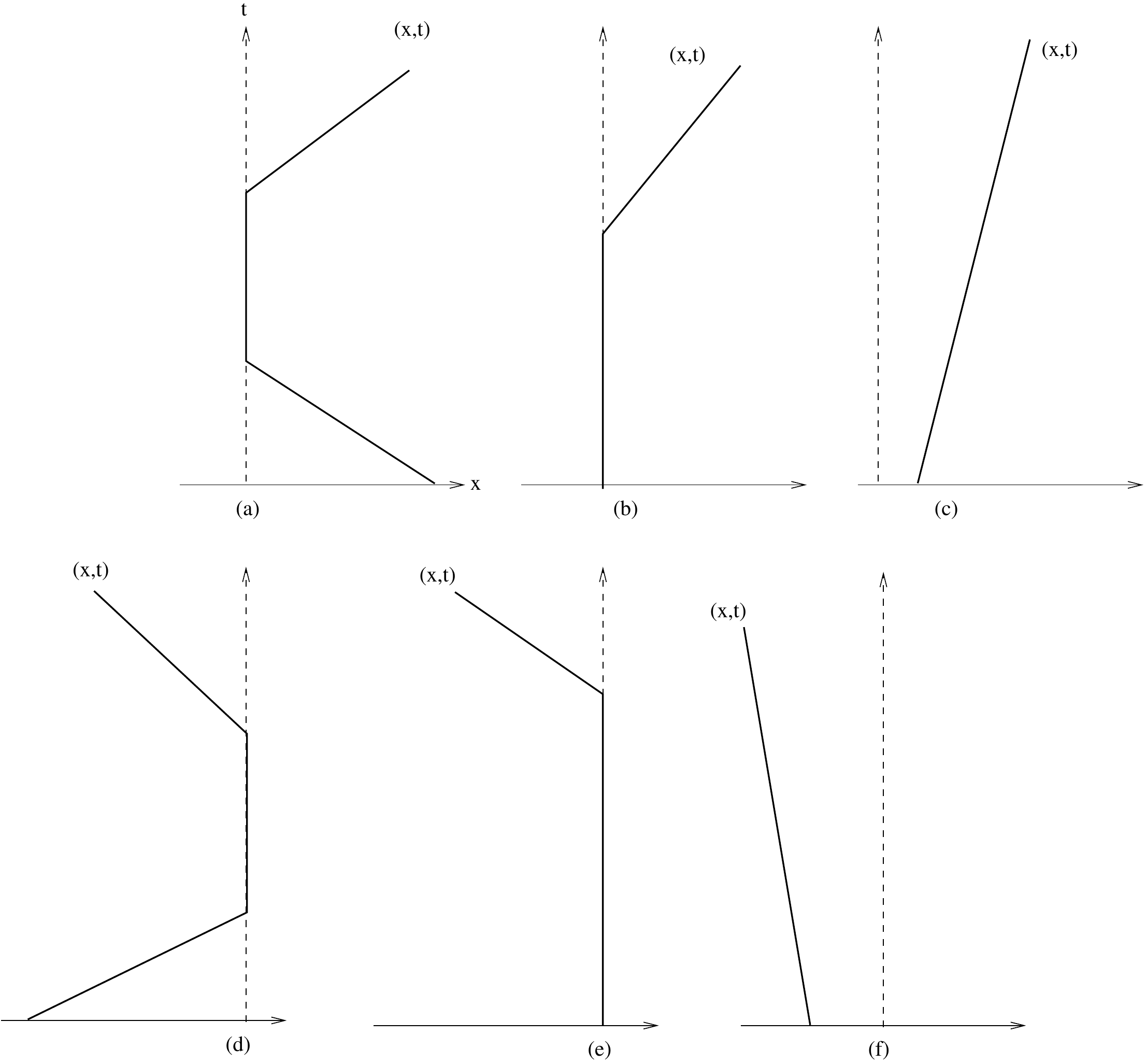}
\end{center}
\caption{Control curves: fig (a), (b), (c) are  positive control curves and fig (d), (e), (f) are negative control curves respectively.}
\end{figure}
Let us denote the set of control curves by $\Gamma(x,t).$

\noindent Positive control curve: Let $x\geq 0$ then define the positive control curve $\Gamma_+(x,t)$ by
\begin{eqnarray*}
 \Gamma_+(x,t)=\{\gamma\in \Gamma(x,t)\ : \ \gamma\geq 0 \}.
\end{eqnarray*}
\noindent Negative control curve: Let $x\leq 0$ then define the negative control curve $\Gamma_-(x,t)$ by
\begin{eqnarray*}
 \Gamma_-(x,t)=\{\gamma\in \Gamma(x,t)\ : \ \gamma\leq 0 \}.
\end{eqnarray*}
Define $\bar{\Gamma}_\pm$ by 
\begin{eqnarray*}
 \bar{\Gamma}_\pm(x,t)=\big\{ \gamma\in \Gamma_\pm(x,t): \{s : \gamma(s)\neq 0\}=\mbox{the interval }(t_1,t),\ \mbox{for some}\ t_1\leq t\big\}.
\end{eqnarray*}

\end{definition}
For simplicity, whenever we write the general flux $h$, it means that it is either $f$ or $g$ respectively in suitable sense.

\begin{definition}\label{lt} \textbf{Convexity, super linear growth, Legendre transformation and some useful facts}:
Let the flux $h$ be a $C^2$ function and strictly convex, i.e.,  $\forall \ 
a,b\in \mathbb{R}$, and $\forall \ r\in [0,1]$, the following holds
\begin{eqnarray*}
 h(ra+(1-r)b)<rh(a)+(1-r)h(b).
\end{eqnarray*}
Also assume that the flux $h$ satisfy
superlinear growth property, i.e.,
\begin{eqnarray*}
 \lim\limits_{|u|\rightarrow \infty} \frac{h(u)}{|u|}=\infty.
\end{eqnarray*}
Then one can define the Legendre transform $h^*$ associated to $h$ by 
\begin{eqnarray*}
 h^*(p)=\sup_{q}\{pq-h(q)\}.
\end{eqnarray*}
It is easy to check that $h^*$ enjoys the following properties
\begin{enumerate}
 \item $h^*$ is $C^1$ and strictly convex.
 \item $h^*$ has the superlinear growth property.
 \item $h^*{^{\prime}}=(h^\prime)^{-1}.$
 \item $h=h^{**}.$
 \item $h^*(h^\prime(p))=ph^\prime(p)-h(p).$
\item $h(h^{*^{\prime}}(p))=ph^{*^{\prime}}(p)-h^*(p).$
\item In addition, if  $h^{\prime\prime}>\alpha>0,$ for some $\alpha$, then $(h^{\prime})^{-1}$ is Lipschitz continuous with Lipschitz 
constant bounded by $\frac{1}{\alpha}.$
 \end{enumerate}
 
\end{definition}

\begin{definition}
\textbf{Cost functional}: Let $v_0$ be the initial data for (\ref{bv12}) and let $h$ be the flux,
then define the cost functionals
 as follows
 \begin{align}
  J(\gamma,v_0,h)&=v_0(\gamma(0)) +\int\limits_0^th^*(\gamma^\prime(\theta))d\theta,\\
    J_+(\gamma,v_0,h)&=v_0(\gamma(0)) +\int\limits_{\{t\ :\ \gamma(t)>0\}}h^*(\gamma^\prime(\theta))d\theta,\\
  J_-(\gamma,v_0,h)&=v_0(\gamma(0)) +\int\limits_{\{t\ :\ \gamma(t)<0\}}h^*(\gamma^\prime(\theta))d\theta.
 \end{align}
\end{definition}
Let $v_0\in C^1(\mathbb{R}\setminus \{0\})\cup \mbox{Lip}(\mathbb{R})$ and $v_0(0)=0$. Define the following auxiliary functions $b_\pm$  by
\begin{align}
 b_+(t):=b_+(t,v_0,f):=\inf\limits_{\{\gamma\in \Gamma_+(0,t)\}} J(\gamma,v_0,f)\\
 b_-(t):= b_-(t,v_0,g):=\sup\limits_{\{\gamma\in \Gamma_-(0,t)\} }J(\gamma,v_0,g).
\end{align}

\begin{definition}
 \textbf{Characteristics and other related functions}: Let us define the set of characteristics  $ch_\pm$ by 
 \begin{align}
 \begin{array}{lllllllllll}
    ch_\pm(t)&=&\{\gamma\in \Gamma_\pm(0,t)\ :\ b_\pm(t)=J(\gamma,u_0,h)\}.
    \end{array}
 \end{align}
 Define $y_\pm,t_\pm$ by 
 \begin{align}
 \begin{array}{lllllllllll}
y_+(t,u_0,h) =\inf\{\gamma(0)\ : \ \gamma\in ch_+(t)\},\\
y_-(t,u_0,h) =\sup\{\gamma(0)\ : \ \gamma\in ch_-(t)\},\\
y_+(x,t)=\mbox{min}\{\gamma(0)\ : \ \gamma\in ch_+(t)\},\\
y_-(x,t)=\mbox{max}\{\gamma(0)\ : \ \gamma\in ch_-(t)\},\\
t_+(x,t)=\mbox{max}\{t_1(\gamma)\ : \ \gamma\in ch_+(t)\},\\
t_-(x,t)=\mbox{min}\{t_1(\gamma)\ : \ \gamma\in ch_-(t)\},\\
R(t)=\mbox{min}\{x \ : \ t_+(x,t)=0\},\\
L(t)=\mbox{max}\{x \ : \ t_-(x,t)=0\},
 \end{array}
 \end{align}
where $t\mapsto y_-(t,u_0,h),\ t\mapsto y_+(t,u_0,h)$ are non-increasing and non-decreasing functions respectively.
For a fix $t>0,$ $x\mapsto y_\pm(x,t), \ t_-(x,t)$ are non-decreasing functions and $x\mapsto t_+(x,t)$ is non-increasing function.
The functions $t\mapsto R(t),\ t\mapsto L(t)$ are Lipschitz continuous functions and there exists some constant $C>0$, such that
$Ct\geq R(t)\geq0, \ -Ct\leq L(t)\leq 0.$ 

 \end{definition}

\begin{definition}
 \textbf{Boundary data}: Let us define the boundary data $\lambda_\pm$ by:
 
 \begin{align}
  \lambda_+(t) &=\left\{\begin{array}{lllll}f_-^{-1}(-b^\prime_+(t)) & \mbox{if}\ -b^\prime_+(t)>\mbox{max}(-b^\prime_-(t),f(B))\\
  f_+^{-1}(\mbox{max}(-b^\prime_-(t),f(B))) & \mbox{if}\ -b^\prime_+(t)\leq \mbox{max}(-b^\prime_-(t),f(B))\end{array}\right.\\
  \lambda_-(t)&=\left\{\begin{array}{lllll}g_+^{-1} (-b^\prime_-(t)) & \mbox{if}\ -b^\prime_-(t)\geq \mbox{max}(-b^\prime_+(t),g(A))\\
                  g_-^{-1} (\mbox{max}(-b^\prime_+(t),g(A))) &\mbox{if}\ -b^\prime_-(t)<\mbox{max}(-b^\prime_+(t),g(A)).
                  \end{array}\right.
 \end{align}
\end{definition}

\begin{definition} 
\noindent \textbf{Value functions}: Let us define the value functions $v_\pm$ by:
 \begin{align}
      v_-(x,t) =\sup_{\gamma\in \Gamma_-(x,t)}\left[J_-(\gamma,v_0,g)-\int_{\gamma=0}g(\lambda_-(\theta))d\theta\right] \ \mbox{if} \ x\leq 0,\\
            v_+(x,t) =\sup_{\gamma\in \Gamma_+(x,t)}\left[J_+(\gamma,v_0,f)-\int_{\gamma=0}f(\lambda_+(\theta))d\theta\right] \ \mbox{if} \ x\geq 0.
 \end{align}

 


\end{definition}

\begin{theorem} (Adimurthi et. al. \cite{jde}) Let the fluxes $f,g$ be strictly convex, smooth and of superlinear growth. Then 
 \begin{enumerate}
\item $f(\lambda_+(t))=g(\lambda_-(t))$.

\item The value functions $v_\pm$ are Lipschitz continuous and $v_+(0,t)=v_-(0,t)$, for all $t>0$. $v_+,v_-$ are the solution of (\ref{bv12}).

\item Define the following Lipschitz continuous function $v$ by 
\begin{align}
 v(x,t)=\left\{\begin{array}{lllll} v_-(x,t) \ \mbox{if}\ x<0, \ t>0,\\
                  v_+(x,t) \ \mbox{if}\ x>0, \ t>0.
                  \end{array}\right.
\end{align}
Then $u=v_x$ is the weak entropy solution of (\ref{bv11}) and satisfies the interface entropy condition (\ref{interface}) 
and the interior entropy condition (\ref{bv23}).  
 \end{enumerate}

\end{theorem}

\noindent\textbf{Some useful formulas}: The case when $A=\theta_g$ or $B=\theta_f$. At the points of differentiability of $t_\pm,y_\pm$, we have 
\begin{eqnarray}
 u(x,t)=
\left\{\begin{array}{lll} (f^{\prime})^{-1}\left(\frac{x-y_+(x,t)}{t}\right)=u_0(y_+(x,t)) &\mbox{ if }& 0\leq R(t)<x<\infty,\\
(f^{\prime})^{-1}\left(\frac{x}{t-t_+(x,t)}\right) &\mbox{ if }& 0\leq x<R(t)\\
=u(0+,t_+(x,t))\\
=f_+^{-1}g(u(0-,t_+(x,t)))\\
=f_+^{-1}g(g^{\prime})^{-1}\left(\frac{-y_-(0-,t_+(x,t))}{t_+(x,t)}\right)\\
=f_+^{-1}g(u_0(y_-(0-,t_+(x,t)))).
\end{array}\right.\label{bv624}
\end{eqnarray}

\begin{eqnarray}
 u(x,t)=
\left\{\begin{array}{lll} (g^{\prime})^{-1}\left(\frac{x-y_-(x,t)}{t}\right)=u_0(y_-(x,t)) &\mbox{ if }& -\infty<x\leq L(t)\leq 0, \\
(g^{\prime})^{-1}\left(\frac{x}{t-t_-(x,t)}\right) &\mbox{ if}& L(t)<x<0,\\
=u(0-,t_-(x,t))\\
=g_-^{-1}f(u(0+,t_-(x,t)))\\
=g_-^{-1}f(f^{\prime})^{-1}\left(\frac{-y_+(0+,t_-(x,t))}{t_-(x,t)}\right)\\
=g_-^{-1}f(u_0(y_+(0+,t_-(x,t)))).
\end{array}\right.\label{bv625}
\end{eqnarray}

The case when $A\neq\theta_g$ and  $B\neq\theta_f$. At the points of differentiability of $t_\pm,y_\pm$, we have 
\begin{eqnarray}
 u(x,t)=
\left\{\begin{array}{lll} (f^{\prime})^{-1}\left(\frac{x-y_+(x,t)}{t}\right)=u_0(y_+(x,t) &\mbox{ if }& 0\leq R(t)<x<\infty,\\
(f^{\prime})^{-1}\left(\frac{x}{t-t_+(x,t)}\right) &\mbox{ if }& 0\leq x<R(t)\\
=u(0-,t_+(x,t))\\
=\lambda_+(t_+(x,t))\\
=f_+^{-1}(\mbox{max}(-b^\prime_-(t_+(x,t)),f(B))).\\
\end{array}\right.\label{bv24}
\end{eqnarray}

\begin{eqnarray}
 u(x,t)=
\left\{\begin{array}{lll} (g^{\prime})^{-1}\left(\frac{x-y_-(x,t)}{t}\right)=u_0(y_-(x,t)) &\mbox{ if }& -\infty<x\leq L(t)\leq 0, \\
(g^{\prime})^{-1}\left(\frac{x}{t-t_-(x,t)}\right) &\mbox{ if}& L(t)<x<0,\\
=u(0-,t_-(x,t))\\
=\lambda_-(t_-(x,t))\\
=g_-^{-1}(\mbox{max}(-b^\prime_+(t_-(x,t)),g(A))),
\end{array}\right.\label{bv25}
\end{eqnarray}
where $b^\prime_\pm$ satisfies the following 
\begin{align}
 b^\prime_-(t)=\left\{\begin{array}{lll}
\displaystyle-g\left((g^\prime)^{-1}\left(-\frac{y_-(t,u_0,g)}{t}\right)\right)  &\mbox{if}&y_-(t,u_0,g)<0,\\
-g(\theta_g) &\mbox{if}& y_-(t,u_0,g)=0,
                  \end{array}\right.
\end{align}

\begin{align}
 b^\prime_+(t)=\left\{\begin{array}{lll}
\displaystyle-f\left((f^\prime)^{-1}\left(-\frac{y_+(t,u_0,g)}{t}\right)\right)  &\mbox{if}&y_+(t,u_0,f)>0,\\
-f(\theta_f) &\mbox{if}& y_+(t,u_0,f)=0.
                  \end{array}\right.
\end{align}

\begin{theorem}
 (Adimurthi et.al \cite{ARGS}). Let $u_0\in L^{\infty}(\mathbb{R})$ and $u$ be the solution of (\ref{bv11}).
 Let $\ t>0, \ \epsilon>0, M>\epsilon, 
\ I(M,\epsilon)=\{x:\epsilon\leq|x|\leq M\}$. Then

\begin{enumerate}
 \item[(1).]  Suppose there exists an $\alpha>0$ such that $f''\geq\alpha, \ g''\geq\alpha$, then there exist $C=C(\epsilon,M,\alpha)$ such that 
\begin{displaymath}
 \begin{array}{lll}
 \textrm{TV}\big(u(\cdot,t),I(M,\epsilon)\big)\leq C(\epsilon,M,t).
\end{array}
\end{displaymath}
\item[(2).] Suppose $u_0\in\textrm{BV}$, and $T>0$. Then there exists $C(\epsilon,T)$ such that for all $0<t\leq T$
\begin{displaymath}
 \textrm{TV}\big(u(\cdot,t),|x|>\epsilon\big)\leq C(\epsilon,t)\textrm{TV}(u_0)+4||u_0||_{\infty}.
\end{displaymath}
\item[(3).] Let $u_0\in\textrm{BV}$, $T>0$ and $A\neq\theta_g$ and $B\neq\theta_f$. 
Then there exists $C>0$ such that for all $0<t\leq T$,
\begin{displaymath}
 \textrm{TV}\big(u(\cdot,t)\big)\leq C \ \textrm{TV}(u_0)+6||u_0||_{\infty}.
\end{displaymath}
\item[(4).] Let $ u_0, \ f_+^{-1}(g(u_0)), \  g_-^{-1}(f(u_0))\in\textrm{BV}$, $T>0$ and $A=\theta_g$. Then for all
 $0<t\leq T$,
\begin{displaymath}
 \begin{array}{lll}
 \textrm{TV}\big(u(\cdot,t)\big)
\leq \textrm{TV}(u_0)+\displaystyle \max\big(TV(f_+^{-1}(g(u_0))),\textrm{TV}(g_-^{-1}(f(u_0)))\big)+6||u_0||_{\infty}.\\\\
\end{array}
\end{displaymath}
\item[(5).] For a certain choice of fluxes $f$ and $g$ there exists $u_0\in \textrm{BV}\cap L^{\infty}$
 such that TV$(u(\cdot,1))=\infty$ if $A=\theta_g$ 
or $B=\theta_f$.
\end{enumerate}
\end{theorem}

\begin{theorem}\label{th1}({Ghoshal \cite{ghoshaljde}}).
Let $u_0\in L^{\infty}(\mathbb{R})$ and $u$ be a solution of (\ref{bv11}). Let
$t>0,\ \epsilon>0,\ M>\epsilon $ and 
\begin{eqnarray*}I(M)&=&\{x\ :\ |x|<M\},\\
I(R_1(t))&=&\{x>0\ :\ x<R_1(t)\},\ I(L_1(t))=\{x<0\ :\ x>L_1(t)\}.\end{eqnarray*}
\begin{enumerate}
 \item [(i).] Let $f(\theta_f)\neq g(\theta_g)$ and $f^{\prime\prime}\geq \alpha,\ g^{\prime\prime}\geq \alpha,$
for some $\alpha>0$, also assuming the fact that $\mbox{Supp}\ u_0\subset[-K,K],$ for
some $K>0,$ then there exists a $T_0>0$ such that \mbox{for all} $t>T_0,$
\begin{eqnarray}
 TV\left(u(\cdotp,t),I(M)\right)\leq C(M,t).\label{r3.1}
\end{eqnarray}
As a consequence we have, \mbox{for all} $t>T_0,$
\begin{eqnarray}
TV(u(\cdotp,t), \mathbb{R})\leq C(t),\label{r3.2}
\end{eqnarray}
where $C(t),C(M,t)>0$ are some constants.
 
 \item [(ii).] Let $f(\theta_f)= g(\theta_g)$ then for all $t>0,$
\begin{eqnarray}
 TV\left(u(\cdotp,t),I(R_1(t))\cup I(L_1(t))\right)\leq C(t).\label{r3.3}
\end{eqnarray}
In addition if  $f^{\prime\prime}\geq \alpha,\ g^{\prime\prime}\geq \alpha,$ for some $\alpha>0,$
 then for all $t>0,$
\begin{eqnarray}
 TV\left(u(\cdotp,t),I(M)\right)\leq C(M,t).\label{r3.4}
\end{eqnarray}
As a consequence, if  $\mbox{Supp}\ u_0\subset[-K,K],$ for
some $K>0,$ then for all $t>0,$
\begin{eqnarray}
TV(u(\cdotp,t), \mathbb{R})\leq C(t).\label{r3.5}
\end{eqnarray}
\item[(iii).] Let $f(\theta_f)=g(\theta_g)$ 
and  $u_0\in BV(\mathbb{R})$ then \mbox{for all} $t>0,$
\begin{eqnarray}
 TV\left(u(\cdotp,t)\right)\leq C(t)(TV(u_0)+1)+4\|u_0\|\infty.\label{r3.6}
\end{eqnarray}

\item[(iv).]
 Let $A\neq \theta_g$ and $B\neq \theta_f$. If $f^{\prime\prime},g^{\prime\prime}\geq\alpha>0,u_0\in L^\infty(\mathbb{R})$  then 
for all $t>0$,
\begin{eqnarray}
TV\left(u(\cdotp,t),I(\epsilon)\cup I(M,\epsilon)\right)\leq C_1(\epsilon)+C_2(\epsilon,M,t).
\end{eqnarray}
As a consequence, if   $u_0\in BV(\mathbb{R})$ then for all $t>0,$
\begin{eqnarray}
 TV\left(u(\cdotp,t)\right)\leq C(\epsilon,t)(TV(u_0)+1)+4\|u_0\|\infty.
\end{eqnarray}
\end{enumerate}
\textbf{Counter example} (Ghoshal \cite{ghoshaljde}):  Let $f(u)=(u-1)^2-1, \ g(u)=u^2$. Then there exists
an initial data $u_0\in BV$ and a sequence $T_n,$ such that 
$\lim\limits_{n\rightarrow\infty} T_n=\infty$ and 
\begin{eqnarray}
 TV\left(u(\cdotp,T_n)\right)=\infty, \ \mbox{for all}\ n.
\end{eqnarray}
\end{theorem}

\bigskip


\section{Main results}

In the following Theorems we have relaxed the assumptions of uniform convexity of the fluxes made in \cite{ARGS, ghoshaljde} and 
allowed  $u_0\in L^\infty$.
Let us denote $I(R(t))=\{x>0\ :\ x<R(t)\},\ I(L(t))=\{x<0\ :\ x>L(t)\}.$\\

\noindent \textbf{Hypothesis on the fluxes $\boldsymbol f,\boldsymbol g$}: Let the fluxes satisfy the following

 \noindent $\boldsymbol H\boldsymbol 1$. $f,g$ be $C^2$, strictly convex and of superlinear growth (see definition \ref{lt}).\\
\noindent $\boldsymbol H\boldsymbol 2$. Either $f^{\prime\prime}>\alpha>0$, for some $\alpha$ or  the zero
of $f^\prime$ and $f^{\prime\prime}$ are the same, i.e.,
 if there exists $p\in\mathbb{R}$ such that $f^{\prime\prime}(p)=0$, then $f^{\prime}(p)=0.$\\
 \noindent $\boldsymbol H\boldsymbol 3$. Either $g^{\prime\prime}>\alpha>0$, for some $\alpha$ or 
 the zero of $g^\prime$ and $g^{\prime\prime}$ are the same.

\begin{theorem}{\label{m2}}
 Let $u_0\in L^\infty$. Let the fluxes satisfy $\boldsymbol H\boldsymbol 1$, $\boldsymbol H\boldsymbol 2$ and $\boldsymbol H\boldsymbol 3$ as above.
 Let the connection satisfies $A\neq \theta_g, B\neq \theta_f$, then 
   \begin{eqnarray}
TV(u(\cdotp,t), I(R(t))\cup I(L(t)))\leq C(t).\label{bv430}
\end{eqnarray}\end{theorem}

\begin{theorem}{\label{m1}}
 Let $u_0\in L^\infty$. Let the fluxes satisfy $\boldsymbol H\boldsymbol 1$, $\boldsymbol H\boldsymbol 2$ and $\boldsymbol H\boldsymbol 3$ as above.
 Let the connection satisfies $A=\theta_g$ or  $B= \theta_f$. If
 $f(\theta_f)\neq g(\theta_g),f(\theta_f)\neq g(0),f(0)\neq g(\theta_g),f(0)\neq g(0)$ and $\mbox{Supp}\ u_0 \subset [-M,M]$, for some $M>0,$
 then there exists $T>0$ such that for all $t>T,$
   \begin{eqnarray}
TV(u(\cdotp,t), I(R(t))\cup I(L(t)))\leq C(t).\label{bv531}
\end{eqnarray}\end{theorem}

\begin{proof} (Proof of Theorem \ref{m2}).
 We consider the following three cases. \\
Case 1: $R(t)>0, L(t)=0$.\\
Case 2: $R(t)=0, L(t)<0$.\\
Case 3: $R(t)>0,L(t)<0$.\\

\noindent Case 1: Since the characteristics speed is positive, for $x\in(0,R(t))$, we have 
\begin{align}
 u(x,t)=&(f^\prime)^{-1}\left(\frac{x}{t-t_+(x,t)}\right)\label{bv431}\\
 =& \lambda_+(t_+(x,t))\label{bv432}\\
 =& f^{-1}_+(\mbox{max}(-b^\prime_-(t_+(x,t)),f(B))).\label{bv433}
\end{align}
Now $b^\prime_-$ satisfies the following relations
\begin{align}
 b^\prime_-(t_+(x,t))=\left\{\begin{array}{lll}
\displaystyle-g\left((g^\prime)^{-1}\left(-\frac{y_-(t_+(x,t),u_0,g)}{t_+(x,t)}\right)\right)  &\mbox{if}&y_-(t_+(x,t),u_0,g)<0,\\
-g(\theta_g) &\mbox{if}& y_-(t_+(x,t),u_0,g)=0,
                  \end{array}\right.\label{bv434}
\end{align}
where $t\mapsto y_-(t,u_0,g)$ is a non-increasing function.

Let us define $\epsilon(t)=\sup\{ x>0 \ : y_-(t_+(x,t),u_0,g)<0\}$. If $\epsilon(t)=0$, then from (\ref{bv432}), (\ref{bv433}) and (\ref{bv434}), we have
$u(x,t)=f^{-1}_+(\mbox{max}(g(\theta_g),f(B)))=\mbox{constant}$, for all $x\in (0,R(t))$, hence in $u\in BV(0,R(t)).$ Now we assume that 
$\epsilon(t)>0.$ Due to the monotonicity of $y_-$, for all $x\in(0,\epsilon(t))$
\begin{align}
 y_-(t_+(0+,t),u_0,g)\leq y_-(t_+(x,t),t),u_0,g)\leq y_-(t_+(\epsilon(t),t),u_0,g)<0. \label{bv435}
\end{align}
 First we prove the result in $(0,\epsilon(t))$ then in $(\epsilon(t),R(t)).$ 
 From the monotonicity of $y_-,t_+$, we conclude 
 \begin{align}
\begin{array}{llllllll}
    \displaystyle \frac{-y_-(t_+(0+,t),u_0,g)}{t_+(\epsilon(t),t)}&\geq& \displaystyle\frac{-y_-(t_+(x,t),u_0,g)}{t_+(\epsilon(t),t)}\\
   &\geq&\displaystyle \frac{-y_-(t_+(x,t),u_0,g)}{t_+(x,t)}\\
   &\geq& \displaystyle\frac{-y_-(t_+(\epsilon(t),t),u_0,g)}{t_+(x,t)}\\&\geq& \displaystyle\frac{-y_-(t_+(\epsilon(t),t),u_0,g)}{t}>0.
  \end{array}\label{bv436}
 \end{align}
From (\ref{bv436}), $\displaystyle \frac{-y_-(t_+(x,t),u_0,g)}{t_+(x,t)}$ is away from $0$, hence
\begin{align}
 (g^\prime)^{-1} \ \mbox{is Lipschitz continuous in} \ \left(\displaystyle\frac{-y_-(t_+(\epsilon(t),t),u_0,g)}{t},
 \displaystyle \frac{-y_-(t_+(0+,t),u_0,g)}{t_+(\epsilon(t),t)}
 \right).\label{bv437}
 \end{align}
  Let us choose a partition $\{x_i\}^{N}_{i=1}$ in $(0,\epsilon(t))$. In view of the fact that $t_+,y_-$ are monotone and using 
(\ref{bv435}), (\ref{bv436}) to obtain
\begin{equation}
 \begin{array}{lllllllllllll}
  \displaystyle\sum\limits_{i=1}^{N}\left|\frac{-y_-(t_+(x_{i+1},t),u_0,g)}{t_+(x_{i+1},t)}+\frac{y_-(t_+(x_i,t),u_0,g)}{t_+(x_i,t)}\right|\\
   \leq \displaystyle\sum\limits_{i=1}^{N}\frac{|y_-(t_+(x_{i+1},t),u_0,g)||t_+(x_{i},t)-t_+(x_{i+1},t)|}
  {|t_+(x_{i},t)t_+(x_{i+1},t)|}\\
  + \displaystyle\sum\limits_{i=1}^{N}\frac{|t_+(x_{i+1},t)|
  |y_-(t_+(x_{i+1},t),u_0,g)-y_-(t_+(x_{i},t),u_0,g)| }{|t_+(x_{i},t)t_+(x_{i+1},t)|}\\
  \displaystyle\leq \frac{|y_-(t_+(0+,t),u_0,g)||t-t_+(\epsilon(t),t)|}{|t_+(\epsilon(t),t)|^2}+\displaystyle\frac{t |y_-(t_+(0+,t),u_0,g)|}{|t_+(\epsilon(t),t)|^2}\\
  \displaystyle\leq \frac{2t|y_-(t_+(0+,t),u_0,g)|}{|t_+(\epsilon(t),t)|^2}.
 \end{array}\label{bv438}
\end{equation}
Since $B\neq \theta_f$, we conclude  $f_+(\lambda_+(t_+(x,t)))=(\mbox{max}(-b^\prime_-(t_+(x,t)),f(B)))\geq f(B)>f(\theta_f)$, therefore
\begin{equation}
 f_+^{-1} \ \mbox{is Lipschitz continuous in the interval} \ \big(f(B), \sup_{x\in (0,R(t)) }-b^\prime_-(t_+(x,t))\big).\label{bv440}
\end{equation}
Since $f,g\in C^2$   using (\ref{bv431}), (\ref{bv432}), (\ref{bv433}), (\ref{bv440}), (\ref{bv438}), we get 
\begin{align}
 \begin{array}{lllllllllllll}
  \displaystyle\sum\limits_{i=1}^{N}|u(x_{i+1},t)- u(x_{i},t)|\\
   = \displaystyle\sum\limits_{i=1}^{N}\left|f^{-1}_+(\mbox{max}(-b^\prime_-(t_+(x_{i+1},t)),f(B)))
   -f^{-1}_+(\mbox{max}(-b^\prime_-(t_+(x_i,t)),f(B)))\right|\\
   \leq C(t)\displaystyle\sum\limits_{i=1}^{N}\left|\mbox{max}(-b^\prime_-(t_+(x_{i+1},t)),f(B))
   -\mbox{max}(-b^\prime_-(t_+(x_i,t)),f(B))\right|\\
  \leq \displaystyle C(t)\sum\limits_{i=1}^{N}\left|b^\prime_-(t_+(x_{i+1},t))-b^\prime_-(t_+(x_i,t))\right|\\
  =\displaystyle C(t)\sum\limits_{i=1}^{N}\left|\frac{-y_-(t_+(x_{i+1},t),u_0,g)}{t_+(x_{i+1},t)}+\frac{y_-(t_+(x_i,t),u_0,g)}{t_+(x_i,t)}\right|\\
  \displaystyle\leq  C(t) \frac{2t|y_-(t_+(0+,t),u_0,g)|}{|t_+(\epsilon(t),t)|^2}.
 \end{array}\label{bv439}
\end{align}
Hence $u\in BV(0,\epsilon(t))$. Next, we prove the results in $(\epsilon(t),R(t))$. 	
Due to the monotonicity of $t_+$,  for $x\in (\epsilon(t),R(t))$, we have 
\begin{align}\label{bv441}
 \begin{array}{llll}
t_+(\epsilon(t),t)\geq  t_+(x,t)\geq t_+(R(t)-, t)\\ 
  t- t_+(\epsilon(t),t)\leq t- t_+(x,t)\leq t-t_+(R(t)-, t)<t,\\
 \end{array}
\end{align}
hence for some $C>0$,
  \begin{equation}\label{bv442}
   \begin{array}{lll}
   \displaystyle 0<\frac{\epsilon(t)}{t} \leq \frac{\epsilon(t)}{t-t_+(R(t)-, t)} \leq \frac{x}{t-t_+(x, t)}
\displaystyle \leq \frac{R(t)}{t-t_+(x, t)} &\leq& \displaystyle\frac{R(t)}{t-t_+(\epsilon(t), t)}\\ &\leq& \displaystyle \frac{Ct}{t-t_+(\epsilon(t), t)}.
   \end{array}
  \end{equation}
Whence from (\ref{bv442}), $\displaystyle \frac{x}{t-t_+(x, t)}$ is away from $0$, which allows 
\begin{align}\label{bv443}
 (f^\prime)^{-1} \ \mbox{is Lipschitz continuous in}\  \displaystyle\left[\frac{\epsilon(t)}{t}, \frac{Ct}{t-t_+(\epsilon(t), t)}\right]
\end{align}
and the Lipschitz constant depends on $t.$
Let us choose a partition $\{x_i\}^{N}_{i=1}$ in $(\epsilon(t),R(t))$. Because of the monotonicity of $t_+$, (\ref{bv441}) and (\ref{bv442}) 
we observe
\begin{align}\label{bv444}
 \begin{array}{lllllllllllllll}
 \displaystyle\sum\limits_{i=1}^{N} \left|\frac{x_{i+1}}{t-t_+(x_{i+1},t)}-\frac{x_{i}}{t-t_+(x_{i},t)}\right|\\
  \leq \displaystyle\sum\limits_{i=1}^{N} \frac{t|x_{i+1}-x_i| + |x_{i+1}t_+(x_i,t)-x_it_+(x_{i+1},t)|}{(t-t_+(x_{i+1},t))(t-t_+(x_{i},t))}\\
    \leq \displaystyle\sum\limits_{i=1}^{N} \frac{t|x_{i+1}-x_i| + |x_{i+1}t_+(x_i,t)-x_it_+(x_{i+1},t)|}{(t-t_+(\epsilon(t),t))^2}\\
 \leq \displaystyle\sum\limits_{i=1}^{N} \frac{t|x_{i+1}-x_i| +
 |x_{i+1}||t_+(x_i,t)-t_+(x_{i+1},t)|+t_+(x_{i+1},t)|x_{i+1}-x_i|}{(t-t_+(\epsilon(t),t))^2}\\
  \leq \displaystyle \frac{3R(t)t}{(t-t_+(\epsilon(t),t))^2}\\
    \leq \displaystyle \frac{3Ct^2}{(t-t_+(\epsilon(t),t))^2}.
 \end{array}
\end{align}
 (\ref{bv431}), (\ref{bv443}) and (\ref{bv444}) yields
\begin{equation}
 \begin{array}{lllllllllllll}
  \displaystyle\sum\limits_{i=1}^{N}|u(x_{i+1},t)- u(x_{i},t)|\\
  = \displaystyle\sum\limits_{i=1}^{N} \left|(f^\prime)^{-1}\left(\frac{x_{i+1}}{t-t_+(x_{i+1},t)}\right)
  -(f^\prime)^{-1}\left(\frac{x_{i}}{t-t_+(x_{i},t)}\right)\right|\\
  \leq \displaystyle\sum\limits_{i=1}^{N} C(t) \left|\frac{x_{i+1}}{t-t_+(x_{i+1},t)}
  -\frac{x_{i}}{t-t_+(x_{i},t)}\right|\\
    \leq \displaystyle \frac{3CC(t)t^2}{(t-t_+(\epsilon(t),t))^2}.
 \end{array}
\end{equation}
Hence for the Case 1, $u(\cdot,t)\in BV(0,R(t)).$


\noindent Case 2 and Case 3: When $L(t)<0$, one can prove the result exactly as $R(t)>0$. Therefore both the cases follows exactly like as Case 1.
Hence the Theorem.
\end{proof}

\begin{proof} (Proof of Theorem \ref{m1}).
Let us assume that $g(\theta_g)>f(\theta_f)$. Denote $0<\delta_1=g(\theta_g)-f(\theta_f)$. We have to consider the following three cases. 

\noindent Case 1: $R(t)>0, L(t)=0$.\\
\noindent Case 2: $R(t)=0, L(t)<0$.\\
\noindent Case 3: $R(t)>0,L(t)<0$.

\noindent Case 1: In this case,
\begin{equation}\label{bv31}
 u(x,t)=(f^\prime)^{-1}\left(\frac{x}{t-t_+(x,t)}\right)\ \mbox{for}\ x\in(0,R(t)).
\end{equation}

Let $\epsilon(t)>0$ be a small number such that $\epsilon(t)<R(t).$ First we prove the result in $(0,\epsilon(t))$ then in $(\epsilon(t),R(t)).$ Suppose there 
exist constants $T_0>0, C >0$ such that
\begin{equation}\label{bv32}
 |t_+(x,t)|<C, \ \mbox{for}\ t>T_0, x\in(0,R(t)).
\end{equation}
Since for  $t>T_0$, $\epsilon(t),t_+(x,t)$ are bounded, hence  from (\ref{bv31}), (\ref{bv32}) and the fact that the characteristics speed is 
positive for $x\in(0,R(t))$ there exists a small $\delta_2>0$, such that 
\begin{equation}\label{bv33}
 u(x,t)\in [\theta_f,\theta_f+\delta_2], \ \mbox{for} \ t>T_0,x\in (0,\epsilon(t)). 
\end{equation}
One can re-choose $T_0$ large, $\delta_2>0$ small such that $f(\theta_f+\delta_2)<g(\theta_g)$ and  (\ref{bv33}) still holds.
From (\ref{bv33}), it is clear that for  $t>T_0$,
\begin{equation}\label{bv34}
 u(0+,t)\in [\theta_f,\theta_f+\delta] \ \mbox{and so} \ f(u(0+,t))\in [f(\theta_f),g(\theta_g)) 
\end{equation}
 Therefore from (\ref{bv34}), it is easy to see that R-H condition $f(u(0+,t)=g(u(0-,t))$ does not hold for $t>T_0,$ which is a contradiction.
 Hence (\ref{bv32}) is false and so for $ x\in(0,\epsilon(t)),$
 \begin{align}
 \varlimsup\limits_{t\rightarrow \infty} t_+(x,t)=\infty . \label{bv35}
 \end{align}

From the R-H condition and the explicit formulas we have 
\begin{align}
 f(u(x,t))=f(u(0+,t_+(x,t)))&=g(u(0-,t_+(x,t)))\label{bv46}\\
 &=g\left(\left(g^\prime\right)^{-1}\left(-\frac{y_-(0-,t_+(x,t))}{t_+(x,t)}\right)\right)\label{bv36}\\
 &=g(u_0(y_-(0-,t_+(x,t)))).\label{bv37}
\end{align}
\par Now if for some $x_0\in (0,\epsilon(t))$, $y_-(0-,t_+(x,t))<-M,$
then by the monotonicity of $y_-$, $y_-(0-,t_+(x,t))<-M$ for all $x\in (0,x_0).$ Since
$\mbox{Supp}\ u_0\subset [-M,M]$ and (\ref{bv37}), we obtain 
$u(x,t)=f^{-1}g(0)$. Hence choosing $\epsilon(t)=x_0$, one has $u(\cdot,t)\in BV (0,\epsilon(t))$.
\par Let us assume the case when 
\begin{align}
 y_-(0-,t_+(x,t))\in [-M,0] \ \mbox{for all}\ x\in (0,\epsilon(t)).\label{bv38}
\end{align}
 Using (\ref{bv35}), (\ref{bv36}) and (\ref{bv38}),
 it is immediate that $\displaystyle-\frac{y_-(0-,t_+(x,t))}{t_+(x,t)}\rightarrow 0$ as $t\rightarrow\infty$.
 In view of the fact that the characteristic speed in positive in $(0,R(t))$, there exists a small $\delta_3>0$ and a large $T_0$ (re-choosing 
 the previous one), such that  for all $x\in(0,\epsilon(t)), t>T_0$
 \begin{align}
  g\left(\left(g^\prime\right)^{-1}\left(-\frac{y_-(0-,t_+(x,t))}{t_+(x,t)}\right)\right)\in [g(\theta_g), g(\theta_g)+\delta_3]. \label{bv39}
 \end{align}
Thanks to $g(\theta_g)>f(\theta_f)$ and (\ref{bv39}), we get 
\begin{align}\label{bv44}
 f^{-1}\ \mbox{is Lipschitz continuous in} \ [g(\theta_g), g(\theta_g)+\delta_3]
\end{align}
 and the Lipschitz constant depends on $t$.
Since $y_+,t_+$ are monotone functions we have the following relations 
\begin{eqnarray}
  t_+(\epsilon(t),t)\leq t_+(x,t)<t, \ \mbox{for all} \ x\in (0,\epsilon(t))\label{bv40}\\ 
 -M\leq  y_-(0-,t_+(x,t))\leq y_-(0-,t_+(\epsilon(t),t))\leq 0,\ \mbox{for all} \ x\in (0,\epsilon(t)) \label{bv41}\\ 
 0<-\frac{y_-(0-,t_+(\epsilon(t),t))}{t}< -\frac{y_-(0-,t_+(\epsilon(t),t))}{t_+(x,t)}\\\leq -\frac{y_-(0-,t_+(x,t))}{t_+(x,t)}\leq -\frac{y_-(0-,t_+(x,t))}{t_+(\epsilon(t),t)}
 \leq \frac{M}{t_+(\epsilon(t),t)},  \label{bv42}\\ \mbox{for all} \ x\in (0,\epsilon(t)).\nonumber
\end{eqnarray}
Therefore from (\ref{bv42}), it is clear that for a fixed $t>T_0$, $\displaystyle-\frac{y_-(0-,t_+(x,t))}{t_+(x,t)}$ is away from  $0$,
hence
\begin{align}\label{bv45}
 (g^\prime)^{-1}\ \mbox{is Lipschitz continuous
in } \ \displaystyle\left(-\frac{y_-(0-,t_+(\epsilon(t),t))}{t},\frac{M}{t_+(\epsilon(t),t)}\right)
\end{align}
and the Lipschitz constant depends on $t$.
Let us choose a partition $\{x_i\}^{N}_{i=1}$ in $(0,\epsilon(t))$. In view of the fact that $t_+,y_-$ are monotone and using 
(\ref{bv40}), (\ref{bv41}), (\ref{bv42}), to obtain
\begin{equation}
 \begin{array}{lllllllllllll}
  \displaystyle\sum\limits_{i=1}^{N}\left|-\frac{y_-(0-,t_+(x_{i+1},t))}{t_+(x_{i+1},t)}+\frac{y_-(0-,t_+(x_{i},t))}{t_+(x_{i},t)}\right|\\
  \leq \displaystyle\sum\limits_{i=1}^{N}\frac{|y_-(0-,t_+(x_{i+1},t))||t_+(x_{i},t)-t_+(x_{i+1},t)|}{|t_+(x_{i},t)t_+(x_{i+1},t)|}\\
  + \displaystyle\sum\limits_{i=1}^{N}\frac{|t_+(x_{i+1},t)||y_-(0-,t_+(x_{i+1},t))-y_-(0-,t_+(x_{i},t))| }{|t_+(x_{i},t)t_+(x_{i+1},t)|}\\
  \displaystyle\leq \frac{M |t-t_+(\epsilon(t),t)|}{|t_+(\epsilon(t),t)|^2} + \frac{Mt}{|t_+(\epsilon(t),t)|^2}\\
  \displaystyle \leq \frac{2Mt}{|t_+(\epsilon(t),t)|^2}
 \end{array}\label{bv43}
\end{equation}
As a results of (\ref{bv46}), (\ref{bv36}), (\ref{bv44}), (\ref{bv45}) and (\ref{bv43}), we get 
\begin{equation}
 \begin{array}{lllllllllllll}
  \sum\limits_{i=1}^{N}|u(x_{i+1},t)- u(x_{i},t)|\\
  = \sum\limits_{i=1}^{N}\left|f^{-1}g\left(\left(g^\prime\right)^{-1}\left(-\frac{y_-(0-,t_+(x_{i+1},t))}{t_+(x_{i+1},t)}\right)\right)
  -f^{-1}g\left(\left(g^\prime\right)^{-1}\left(-\frac{y_-(0-,t_+(x_{i},t))}{t_+(x_i,t)}\right)\right)\right|\\
   \leq  C_1(t)\sum\limits_{i=1}^{N}\left|g\left(\left(g^\prime\right)^{-1}\left(-\frac{y_-(0-,t_+(x_{i+1},t))}{t_+(x_{i+1},t)}\right)\right)
  -g\left(\left(g^\prime\right)^{-1}\left(-\frac{y_-(0-,t_+(x_{i},t))}{t_+(x_i,t)}\right)\right)\right|\\
  \leq  C_2 C_1(t)\sum\limits_{i=1}^{N}\left|(g^\prime)^{-1}\left(-\frac{y_-(0-,t_+(x_{i+1},t))}{t_+(x_{i+1},t)}\right)
  -(g^\prime)^{-1}\left(-\frac{y_-(0-,t_+(x_{i},t))}{t_+(x_i,t)}\right)\right|\\
  \leq  C_3(t)C_2 C_1(t)\sum\limits_{i=1}^{N}\left|\left(-\frac{y_-(0-,t_+(x_{i+1},t))}{t_+(x_{i+1},t)}\right)
  -\left(-\frac{y_-(0-,t_+(x_{i},t))}{t_+(x_i,t)}\right)\right|\\
  \leq  C_3(t)C_2 C_1(t) \frac{2M |t-t_+(\epsilon(t),t)|}{|t_+(\epsilon(t),t)|^2}.
 \end{array}
\end{equation}
Hence $u\in BV(0,\epsilon(t))$. Next, we prove the results in $(\epsilon(t),R(t))$. 	
Due to the monotonicity of $t_+$,  for $x\in (\epsilon(t),R(t))$, we have 
\begin{align}\label{bv47}
 \begin{array}{llll}
t_+(\epsilon(t),t)\geq  t_+(x,t)\geq t_+(R(t)-, t)\\ 
  t- t_+(\epsilon(t),t)\leq t- t_+(x,t)\leq t-t_+(R(t)-, t)<t.\\
 \end{array}
\end{align}
Whence
  \begin{equation}\label{bv48}
   \begin{array}{lll}
   \displaystyle 0<\frac{\epsilon(t)}{t} < \frac{x}{t-t_+(x, t)}< \displaystyle \frac{C_4t}{t-t_+(\epsilon(t), t)},
   \end{array}
  \end{equation}
  for some constant $C_4>0.$
For a fix $t>T_0$, it is clear from (\ref{bv48}) that $\displaystyle \frac{x}{t-t_+(x, t)}$ is away from $0$, which allows 
\begin{align}\label{bv49}
 (f^\prime)^{-1} \ \mbox{is Lipschitz continuous in}\  \displaystyle\left[\frac{\epsilon(t)}{t}, \frac{C_4t}{t-t_+(\epsilon(t), t)}\right]
\end{align}
and the Lipschitz constant depends on $t.$
Let us choose a partition $\{x_i\}^{N}_{i=1}$ in $(\epsilon(t),R(t))$. Because of the monotonicity of $t_+$, (\ref{bv47}) and (\ref{bv48}) 
we observe
\begin{align}\label{bv50}
 \begin{array}{lllllllllllllll}
 \displaystyle\sum\limits_{i=1}^{N} \left|\frac{x_{i+1}}{t-t_+(x_{i+1},t)}-\frac{x_{i}}{t-t_+(x_{i},t)}\right|\\
     \leq \displaystyle\sum\limits_{i=1}^{N} \frac{t|x_{i+1}-x_i| + |x_{i+1}t_+(x_i,t)-x_it_+(x_{i+1},t)|}{(t-t_+(\epsilon(t),t))^2}\\
 \leq \displaystyle\sum\limits_{i=1}^{N} \frac{t|x_{i+1}-x_i| +
 |x_{i+1}||t_+(x_i,t)-t_+(x_{i+1},t)|+t_+(x_{i+1},t)|x_{i+1}-x_i|}{(t-t_+(\epsilon(t),t))^2}\\
  \leq \displaystyle \frac{3R(t)t}{(t-t_+(\epsilon(t),t))^2}\\
    \leq \displaystyle \frac{3C_4t^2}{(t-t_+(\epsilon(t),t))^2}
 \end{array}
\end{align}
 (\ref{bv31}), (\ref{bv49}) and (\ref{bv50}) yields
\begin{equation}
 \begin{array}{lllllllllllll}\label{bv51}
  \displaystyle\sum\limits_{i=1}^{N}|u(x_{i+1},t)- u(x_{i},t)|\\
  = \displaystyle\sum\limits_{i=1}^{N} \left|(f^\prime)^{-1}\left(\frac{x_{i+1}}{t-t_+(x_{i+1},t)}\right)
  -(f^\prime)^{-1}\left(\frac{x_{i}}{t-t_+(x_{i},t)}\right)\right|\\
  \leq \displaystyle\sum\limits_{i=1}^{N} C_5(t) \left|\frac{x_{i+1}}{t-t_+(x_{i+1},t)}
  -\frac{x_{i}}{t-t_+(x_{i},t)}\right|\\
    \leq \displaystyle \frac{3CC_5(t)t^2}{(t-t_+(\epsilon(t),t))^2}.
 \end{array}
\end{equation}
Hence for the Case 1, $u(\cdot,t)\in BV(0,R(t)).$

\noindent Case 2: For $x\in (L(t),0)$, we have the following 
\begin{align}
 u(x,t)&=\left(g^\prime\right)^{-1}\left(\frac{x}{t-t_-(x,t)}\right)\label{bv52}\\
 &=u(0-,t_-(x,t))\label{bv53}\\
 &=g^{-1}f\left((f^\prime)^{-1}\left(-\frac{y_+(0+,t_-(x,t))}{t_-(x,t)}\right)\right)\label{bv54}\\
 &=g^{-1}f\left((u_0(y_+(x,t))\right)\label{bv55}.
\end{align}
We consider the following four subcases. 

\noindent Subcase I. For all $x\in (L(t),0)$, and  for all $t>T_0$,  $y_+(0+,t_-(x,t))\leq M$, $t_-(x,t)<C$, for some constant $C>0.$

\noindent Subcase II. For some $x_0\in (L(t),0)$,   $y_+(0+,t_-(x_0,t))> M$ and for all $t>T_0$,for all $x\in (L(t),0)$, $t_-(x,t)<C$,  
for some constant $C>0.$

\noindent Subcase III. For some $x_0\in (L(t),0)$,  for all $t>T_0$,  $y_+(0+,t_-(x_0,t))> M$ and for $x\in (L(t),0)$, 
$\varlimsup\limits_{t\rightarrow\infty}t_-(0-,t)=\infty$

\noindent Subcase IV. For all $x\in (L(t),0)$, for all $t>T_0$,  $y_+(0+,t_-(x,t))\leq M$ and  $\varlimsup\limits_{t\rightarrow\infty}t_-(0-,t)=\infty$.

\noindent Subcase I:  Since $t_-(x,t)<C$, whence for $t>T_0,$ we obtain
\begin{align}\label{bv56}
 u(0-,t)=\displaystyle \lim\limits_{x\rightarrow 0-} u(x,t)= \displaystyle \lim\limits_{x\rightarrow 0-} 
 \left(g^\prime\right)^{-1}\left(\frac{x}{t-t_-(x,t)}\right) =\left(g^\prime\right)^{-1}\left(0\right)=\theta_g.
\end{align}

In this Case 2, $R(t)=0.$ For $x>0,$ we have $\displaystyle u(x,t)=\left(f^\prime\right)^{-1}\left(\frac{x-y_+(x,t)}{t}\right)$. 

If  $y_+(0+,t)>M$ then due to the compact support of $u_0$, and monotonicity of $y_+$, 
\begin{align}\label{bv57}
 u(0+,t)=\displaystyle \lim\limits_{x\rightarrow 0+} u(x,t)= \displaystyle \lim\limits_{x\rightarrow 0+} 0=0. 
\end{align}
As a result of R-H condition, (\ref{bv56}) and (\ref{bv57}), we obtain $f(0)=g(\theta_g),$ which contradicts the assumption $f(0)\neq g(\theta_g).$
 \par If $y_+(0+,t)<M$, then there exists a large $T_0>0$ and a small $\delta>0$
 such that for $x\in (0,\epsilon(t))$
 \begin{align}\label{bv58}
 u(0+,t) \in [\theta_f-\delta,\theta_f].
 \end{align}
Since $\delta>0$ is small and $g(\theta_g)>f(\theta_f)$, (\ref{bv58}) violets the R-H condition $f(u(0+,t))=g(u(0-,t))$. Therefore the Subcase I,
can never occur. \\

\noindent SubCase II: Since for some $x_0\in (L(t),0)$, $y_+(0+,t_-(x_0,t))>M$. From the monotonicity of $y_+$ and Supp $u_0\subset [-M,M]$, we deduce
$u(x,t)=g^{-1}f(0),$ for $x\in(x_0,0)$, therefore $u\in BV(x_0,0)$. Then one can repeat the same argument as in (\ref{bv51}) 
to prove $u\in BV(L(t),x_0).$\\

\noindent Subcase III: Follows as in Subcase II.\\

\noindent Subcase IV: Since $y_+(0+,t_-(x,t))\leq M$, for $x\in (L(t),0),$ therefore from (\ref{bv54}) and similarly as in (\ref{bv58}), it contradicts
the R-H condition. Therefore the Subcase IV, can never occur.
\par This proves Case 2.

\noindent Case 3: This case is not allowed due to the interface entropy condition (\ref{interface2}). 

\par Similarly one can repeat the arguments above for the case when $f(\theta_f)>g(\theta_g).$

Hence the Theorem. 

\end{proof}



\end{document}